\documentclass{amsart}
\usepackage{amsmath,amsthm,amssymb,amsrefs}

\theoremstyle{plain}
\newtheorem{thm}{Theorem}[section]
\newtheorem{lemma}[thm]{Lemma}

\theoremstyle{definition}
\newtheorem{defn}[thm]{Definition}

\newcommand{\union}{\cup}
\newcommand{\bigunion}{\bigcup}
\newcommand{\inters}{\cap}

\newcommand{\R}{\mathbb{R}}
\newcommand{\F}{\mathbb{F}}
\newcommand{\N}{\mathbb{N}}
\newcommand{\G}{\mathrel{G}}
\newcommand{\Gpdx}{G_{\text{p}}}

\newcommand{\define}[1]{\emph{#1}}
\newcommand{\from}{\colon}
\renewcommand{\subset}{\subseteq}

\newcommand{\Nbhd}{\mathrm{N}}

\newcommand{\card}[1]{\lvert #1 \rvert}
\newcommand{\Z}{\mathbb{Z}}

\DeclareMathOperator{\Hall}{Hall}
\DeclareMathOperator{\ran}{ran}

\begin{document}

\title{Baire measurable paradoxical decompositions via matchings}

\thanks{The first author was partially supported by NSF grants DMS-1204907
and DMS-1500974 and the John Templeton Foundation under Award No.\ 15619.}

\author{Andrew Marks}
\address{Department of Mathematics, University of California at Los Angeles}
\email{marks@math.ucla.edu}

\author{Spencer Unger}
\address{Department of Mathematics, University of California at Los Angeles}
\email{sunger@math.ucla.edu}

\begin{abstract}
  We show that every locally finite bipartite Borel graph satisfying a
  strengthening of Hall's condition has a Borel perfect matching on some
  comeager invariant Borel set. 
  We apply this to show that if a group
  acting by Borel automorphisms on a Polish space has a paradoxical
  decomposition, then it admits a paradoxical decomposition using pieces
  having the Baire property. This strengthens a theorem of Dougherty and
  Foreman who showed that there is a paradoxical decomposition of the unit
  ball in $\R^3$ using Baire measurable pieces.
  We also obtain a Baire category solution to
  the dynamical von Neumann-Day problem: if $a$ is a nonamenable action of
  a group on a Polish space $X$ by Borel automorphisms, then there is a
  free Baire measurable action of $\F_2$ on $X$ which is Lipschitz with
  respect to $a$. 
\end{abstract}

\maketitle

\section{Introduction}

The Banach-Tarski paradox states that the unit ball in $\R^3$ is
equidecomposable with two unit balls in $\R^3$ by rigid motions. 
In 1930, Marczewski asked whether there is such an equidecomposition where
each piece has the Baire property~\cite{MR1251963}. Using an intricate
construction, Dougherty and Foreman gave a positive answer to this
question~\cites{MR1190902,MR1227475}. 
The key result used in their proof is that every free action of
$\F_2$ on a Polish (separable, completely metrizable) space by homeomorphisms has a paradoxical decomposition using pieces with
the Baire property, where $\F_2$ is the free group on two generators.

Recall that a group $\Gamma$ is said to \define{act by Borel automorphisms}
on a Polish space $X$ if for every $\gamma \in \Gamma$, the function $x
\mapsto \gamma \cdot x$ is Borel. We generalize Dougherty and Foreman's
theorem to completely characterize which group actions by Borel
automorphisms have paradoxical decompositions using pieces with the Baire
property. 

\begin{thm}\label{baire_paradox}
  Suppose a group acts on a Polish space by Borel automorphisms, 
  and this action has a paradoxical decomposition. Then the
  action has a paradoxical decomposition where each piece has the Baire
  property.
\end{thm}

This also yields a new proof of Dougherty and Foreman's theorem that the
unit ball $B$ in $\R^3$ is equidecomposable with two unit balls using
pieces with the Baire property. The action of the group of rotations on
the ball $B \setminus \{\mathbf{0}\}$ without the origin has a paradoxical
decomposition. Since this action is continuous (and hence Borel) and $B
\setminus \{\mathbf{0}\}$ is Polish, by Theorem~\ref{baire_paradox}, this
action has a paradoxical decomposition using pieces with the Baire
property. To finish, we combine this with the easy fact that $B$ is
equidecomposable with $B \setminus \{\mathbf{0}\}$ using pieces which are
Borel (see \cite[Corollary 3.10]{MR1251963}).

In contrast to Theorem~\ref{baire_paradox}, we also show that there is an
action of a finitely generated group on a Polish space $X$ by Borel
automorphisms, and two Borel sets $B_0, B_1 \subset X$ which are
equidecomposable with respect to this action, but not equidecomposable
using pieces with the Baire property. 

An action of a group $\Gamma$ on a space $X$ is said to be
\define{amenable} if $X$ admits a finitely additive probability measure 
which is invariant under the action of $\Gamma$. A discrete group
$\Gamma$ is said to be \define{amenable} if the left translation action of
$\Gamma$ on itself is amenable. The notion of amenability was introduced by
von Neumann in response to the Banach-Tarski paradox as an obstruction to
having a paradoxical decomposition. A theorem of Tarski states that an
action of a group is nonamenable if and only if the action admits a
paradoxical decomposition (see \cite[Corollary 9.2]{MR1251963}).

Since every subgroup of an amenable group is amenable and $\F_2$ is easily
seen to be nonamenable, it is natural to ask whether a group is nonamenable
if and only if it contains $\F_2$ as a subgroup. This problem was first
posed in print by Day~\cite{MR0044031} and is sometimes known as the von
Neumann conjecture. It was answered in the negative by 
Ol{\cprime}shanskii~\cite{MR586204}. Despite this negative answer, 
interesting positive results have been proven
for variants of this question where the role of subgroup is replaced by more
general notions. Whyte has given a positive solution to the von Neumann-Day
problem in the setting of geometric group theory \cite{MR1700742}, and
Gaboriau and Lyons have given a positive solution in the setting of
measurable group theory \cite{MR2534099}. We give the following Baire
category solution to the dynamical von Neumann-Day problem. If $a$ and $b$
are actions of $\Gamma$ and $\Delta$ respectively on a space $X$, then the
action $b$ is said to be \define{$a$-Lipschitz}, if for every $\delta \in
\Delta$, there is a finite set $S \subset \Gamma$ such that $\forall x \in
X \exists \gamma \in S (\delta \cdot_b x = \gamma \cdot_a x)$. Recall that
a function is said to be \define{Baire measurable} if the preimage of every
Borel set has the Baire property.

\begin{thm}\label{dynamical_vnd}
  Suppose $a$ is a nonamenable action of a group $\Gamma$ on a Polish space
  $X$ by Borel automorphisms. 
  Then there is a free $a$-Lipschitz action 
  of $\F_2$ on $X$ by Baire measurable automorphisms. 
\end{thm}

An important tool we use to prove both Theorems~\ref{baire_paradox}
and \ref{dynamical_vnd} is a connection between paradoxical decompositions
and perfect matchings. As we describe in Section~\ref{sec:paradox}, if $a$ is an action of a group $\Gamma$ on $X$, then
to each finite symmetric set $S \subset \Gamma$ we can associate a graph
$\Gpdx(a,S)$ so that $\Gpdx(a,S)$ has a perfect matching if and only if $a$
has a paradoxical decomposition using group elements from $S$. 
We exploit this connection by showing that we can
generically construct perfect matchings for Borel graphs satisfying a
strengthening of Hall's condition, and we can amplify Hall's condition in the graph $\Gpdx(a,S)$
to this stronger condition in $\Gpdx(a,S^2)$.
Recall that a bipartite graph $G$ with
bipartition $\{B_0, B_1\}$ satisfies \define{Hall's condition} if for every
finite subset $F$ of $B_0$ or $B_1$, $\card{\Nbhd_G(F)} \geq \card{F}$,
where $\Nbhd_G(F)$ is the set of vertices adjacent to elements of $F$ that
are not contained in $F$. 
Hall's theorem states that a locally finite bipartite graph has a perfect
matching if and only if it satisfies Hall's condition. 

\begin{thm}\label{bairely_hall}
  Suppose $G$ is a locally finite bipartite Borel graph on a Polish space
  with bipartition $\{B_0, B_1\}$ and there exists an $\epsilon > 0$ such
  that for every finite set $F$ with $F \subset B_0$ or $F \subset B_1$,
  $\card{\Nbhd_G(F)} \geq (1 + \epsilon) \card{F}$. Then there is a Borel
  perfect matching of $G$ restricted to some $G$-invariant comeager Borel set. 
\end{thm}

This theorem is optimal in the sense that the $\epsilon$ cannot be made
equal to $0$. Laczkovich has given an example
of a $2$-regular Borel bipartite graph which satisfies Hall's condition but
has no Borel perfect matching restricted to any comeager
set~\cite{MR947676}. Conley and Kechris have also extended this example to
show that for every even $d$ there is a $d$-regular Borel bipartite graph
which satisfies Hall's condition but has no Borel perfect matching
restricted to any comeager set~\cite{MR3019078}. We note
that while both the theorems of Laczkovich and Conley-Kechris are stated in
the measure theoretic context, their use of ergodicity can be replaced with
generic ergodicity to prove these results as we have stated them. 

In contrast to Theorem~\ref{bairely_hall}, Kechris and Marks have observed
that for every $k \geq 1$, there is a bipartite Borel graph $G$ of bounded
degree on a standard probability space $(X,\mu)$ with Borel bipartition
$\{B_0, B_1\}$ such that for every finite subset $F$ of $B_0$ or $B_1$,
$\card{\Nbhd_G(F)} \geq k \card{F}$, yet $G$ has no Borel perfect matching
restricted to any $\mu$-conull set~\cite{KM}. 

The papers~\cite{CM}, \cite{MR2825538}, and \cite{Marks} also contain some
related work on Borel matchings. See~\cite{KM} for a recent survey of work
on matchings in the setting of Borel graph combinatorics, which also
contains a different proof of Dougherty and Foreman's result. 

The authors would like to thank Clinton Conley, Matt Foreman, Alekos
Kechris, Thomas Sinclair, Simon Thomas, and Robin Tucker-Drob for helpful
conversations. The authors would also like to thank the referee for a 
careful reading and many helpful suggestions. 

\section{Definitions}\label{defns}

Our conventions in graph theory follow~\cite{MR2744811}. A \define{perfect
matching} $M$ of a graph $G$ is a set of edges from $G$ such that each
vertex of $G$ is incident to exactly one edge of $M$. If $F$ is a set of
vertices in a graph $G$, the set of neighbors of points in $F$ that are not
contained in $F$ is denoted $\Nbhd_G(F) = \{y \notin F : \exists x \in F (x
\G y)\}$. A subset of the vertices of a graph $G$ is said to be
\define{$G$-invariant} if it is a union of connected components of $G$. A
graph is \define{locally finite (respectively countable)}  if the degree of
every vertex is finite (respectively countable),
and is \define{$d$-regular} if the degree of every vertex is $d$. The \define{closed
$r$-ball} around a vertex $x$ in $G$ is the set $\{y : d_G(x,y) \leq r\}$,
where $d_G(x,y)$ is the distance from $x$ to $y$ in $G$. If $G$ is a graph
with vertex set $X$ and $A \subset X$, then we use $G \restriction A$ to
denote the induced graph on $A$. So $G \restriction A$ is the graph with
vertex set $A$ whose edges are all the edges in $G$ between elements of
$A$. 
If $G$ is a graph on $X$, then we let $G^{\leq n}$ be the graph on
$X$ where $x \mathrel{G^{\leq n}} y$ if $1 \leq d_G(x,y) \leq n$ and
$G^{n}$ be the graph on $X$ where $x \mathrel{G^n} y$ if $d_G(x,y) = n$.

The descriptive combinatorics of Borel graphs was first
systematically studied by Kechris, Solecki and Todorcevic~\cite{MR1667145}.
A \define{Borel graph} $G$ on a Polish space $X$ (or a standard Borel space
$X$) is a symmetric irreflexive relation on $X$ that is Borel as a subset of
$X \times X$. That is, we will identify the graph with its edge relation.
By \cite[Proposition
4.10]{MR1667145} (which is a corollary of the proof of the Feldman-Moore
theorem~{\cite[Theorem 1.3]{MR2095154}}), if $G$ is a locally countable Borel graph on a Polish space
$X$, then there is a countable set $\langle T_i \mid i \in \N \rangle$ of
Borel automorphisms of $X$ such that $x \G y$ if and only if $x \neq y$ and there is an $i$ such
that $T_i(x) = y$. (Indeed, the $T_i$ may be chosen to be involutions, so
$T^2_i(x) = x$). 
From this, it follows that any property of vertices of a
locally finite Borel graph which only depends on the isomorphism class of
its $r$-balls is Borel. All the Borel graphs used to prove the theorems in
the introduction will come equipped with natural functions generating
their edges. A standard reference for descriptive set theory is
\cite{MR1321597}. 

A \define{paradoxical decomposition} of an action of a group $\Gamma$ on a
space $X$ is a partition of $X$ into finitely many sets $\{A_1, \ldots,
A_n, B_1, \ldots, B_m\}$ such that there exist group elements $\alpha_1,
\ldots, \alpha_n, \beta_1, \ldots, \beta_m \in \Gamma$ such that the sets
$\alpha_i A_i$ are pairwise disjoint, the sets $\beta_i B_i$ are pairwise
disjoint, and $X = \alpha_1 A_1 \union \ldots \union \alpha_n A_n = \beta_1
B_1 \union \ldots \union \beta_m B_m$. A basic reference on paradoxical
decompositions is~\cite{MR1251963}. Note that the action of a group
$\Gamma$ on a space $X$
is paradoxical if and only if there exists a finitely generated subgroup
$\Delta \leq \Gamma$ so that the restriction of the action to $\Delta$ is
paradoxical. So while we typically state our results for actions of
arbitrary groups, it generally suffices to consider actions of
finitely generated groups. 

If $a$ is an action of a group $\Gamma$ on a space $X$, then $A, B \subset
X$ are said to be \define{$a$-equidecomposable} if there is a partition
$\{A_1, \ldots, A_n\}$ of $A$ into finitely many sets and group elements
$\alpha_1, \ldots, \alpha_n \in \Gamma$ such the sets $\alpha_i A_i$ are pairwise
disjoint, and $B = \alpha_1 A_1 \union \ldots \union \alpha_n A_n$.
It is easy to see that equidecomposability is an equivalence relation. In
the language of equidecomposability, $a$ has a
paradoxical decomposition if and only if $X$ can be partitioned into two
sets $\{A,B\}$ such that $X$ is $a$-equidecomposable with both $A$ and $B$.

\section{Baire measurable matchings}

We begin with a lemma which
says roughly that modulo a meager set, the vertices of a Borel graph can
be decomposed into countably many Borel sets where pairs of distinct
vertices in each set have large pairwise distance. 

\begin{lemma} \label{comeager} 
Let $G$ be a locally finite Borel graph on a Polish space $X$ and
$f: \N \to \N$ be a function.  Then there is a sequence $\langle A_n \mid n \in \N
\rangle$ of Borel sets such that $A=
\bigcup_{n \in \N} A_n$ is comeager and $G$-invariant and 
distinct $x,y \in A_n$ have $d_G(x,y) > f(n)$.
\end{lemma}

\begin{proof}
Let $\langle U_i \mid i \in \N \rangle$ enumerate a basis of open sets for
$X$. We define sets $B_{i,r}$ by setting $x \in B_{i,r}$ if and only if $x
\in U_i$ and for all $y \neq x$ in the closed $r$-ball around $x$, we have
$y \notin U_i$. These sets are Borel since we can identify the closed
$r$-ball around $x$ in a Borel way. Note that distinct $x, y \in B_{i,r}$
have $d_G(x,y) > r$. Finally, we claim that for each fixed $r$, $X =
\bigcup_{i} B_{i,r}$.  This is because for each $x \in X$, since $G$ is
locally finite, the elements of the closed $r$-ball around $x$ which are
not equal to $x$ form a finite set. Hence, $x$ can be separated from this
finite set by some $U_i$.

We are now ready to construct our sets $A_n$. Choose a countable set of
Borel automorphisms $S_i \from X \to X$ generating $G$, and let $\langle T_i \mid
i \in \N \rangle$ be the closure of the automorphisms $S_i$ under
composition and inverses. Hence, $x$ and $y$ are in the same connected component of $G$
if and only if there is an $i$ such that $T_i(x) = y$. Fix an enumeration
of $\N^2$. We will choose each $A_n$ such that $T_i(A_n)$ is nonmeager in
$U_j$, where $(i,j)$ is the $n$th pair of natural numbers. For a fixed $n$
and $i$, since $X = \bigcup_{k } B_{k,f(n)}$ and $T_i$ is an
automorphism, we have that $X = \bigcup_{k } T_i(B_{k,f(n)})$.
Thus, we can apply the Baire
Category theorem to find a $k$ such that $T_i(B_{k,f(n)})$ is nonmeager in
$U_j$. Let $A'_n$ be this $B_{k,f(n)}$ and let $A' = \bigcup_n A'_n$. Note
that $A'_n$ has the property that distinct $x,y \in A_n'$ have $d_G(x,y) >
f(n)$.
Now
for every $i$, the set $T_i(A')$ is comeager since it is nonmeager in every
open set. Thus, the intersection $A = \bigcap_i T_i(A')$ is comeager since
it is a countable intersection of comeager sets. Finally, $A$ is a
$G$-invariant set since the $T_i$ generate the connectedness relation of
$G$; if $x \notin A$, then $T_i(x) \notin A$ for all $i$. Our desired sets
are $A_n = A \inters A'_n$.
\end{proof}

We remark that we could give a more compact but less elementary proof
of the above theorem as follows.
For each $r$, fix a
Borel $\N$-coloring $c_r \from X \to \N$ of $G^{\leq r}$
by~{\cite[Proposition 4.5]{MR2095154}}.
Let $B_{i,r} =
c_r^{-1}(i)$. (Indeed, the first paragraph of the proof of
Lemma~\ref{comeager} essentially repeats the proof from~{\cite[Proposition
4.5]{MR2095154}} that the graphs $G^{\leq r}$ have Borel $\N$-colorings).
For each $x \in X$, the set of parameters $p \in \N^\N$ such
that $x \in X_p = \bigunion_{n \in \N} B_{p(n),f(n)}$ is comeager (in fact,
it is dense open). Now let
$Y_p$ be the set of $x \in X$ so that every $y$ in the connected component
of $x$ is contained in $X_p$. So the set of $p \in \N^\N$ for
which $x \in Y_p$ is also comeager. 
By the Kuratwoski-Ulam
Theorem~\cite[Theorem 8.41]{MR1321597}, there are comeagerly many $p \in
\N^\N$ for which $Y_p$ is comeager. Fixing such a $p$, we let our sets
$A_n$ be $Y_p \inters B_{p(n),f(n)}$. 

Throughout the paper, Lemma~\ref{comeager} will be the only way in
which we discard meager sets. 

We are now ready to prove Theorem~\ref{bairely_hall}. We begin by
introducing some notation. If $G$ is a graph and $M$ is a partial matching
of $G$, then $G - M$ is the graph obtained by removing the vertices
incident on edges of $M$ and any edges incident on removed vertices. We
note that if $G$ is a Borel graph and $M$ is Borel, then $G - M$ is a Borel
graph. 

If $G$ is a bipartite graph on $X$, then every finite set $F \subset X$ can be partitioned into
$G^{2}$-connected sets having disjoint sets of neighbors in $G$. 
Thus, to
show that Hall's condition holds for a graph, it suffices to show it holds
just for $G^2$-connected sets. 
We will use the following strengthening of Hall's condition where we
restrict our attention to $G^2$-connected sets of size at least $n$. 
\begin{defn} 
A bipartite graph $G$ with bipartition $\{B_0,B_1\}$ satisfies
$\Hall_{\epsilon,n}$ if $G$ satisfies Hall's condition and additionally for
every finite $G^{2}$-connected set $F$ with $\card{F} \geq n$ and $F
\subset B_0$ or $F \subset B_1$, $\card{\Nbhd_G(F)} \geq (1 +
\epsilon)\card{F}$. 
\end{defn}
If $n = 1$, by decomposing into $G^2$-connected sets with disjoint
neighborhoods, this definition is
equivalent to the simpler requirement that for all finite sets $F$ with $F
\subset B_0$ or $F \subset B_1$, $\card{\Nbhd_G(F)} \geq (1 + \epsilon)
\card{F}$. This is the hypothesis of Theorem~\ref{bairely_hall}.
However, for $n > 1$ it becomes important the we only consider
$G^2$-connected sets in our proof.

\begin{proof}[Proof of Theorem~\ref{bairely_hall}.]
Let $G$ be as in the theorem with $\epsilon$ witnessing that $G$ satisfies
$\Hall_{\epsilon,1}$. Note we do not require that $B_0$ and $B_1$ are
Borel.  Let $f:\N\to \N$ be an increasing function such
that for all $n$, $f(n) \geq 8$ and $\sum_{n \in \N} 8/f(n) < \epsilon$.
Apply Lemma
\ref{comeager} with the function $f$ to obtain an invariant comeager $A = \bigcup_{n \in \N}
A_n$, which is a union of Borel sets $A_n$ of elements of pairwise distance greater
than $f(n)$. We will find a Borel perfect
matching of $G \restriction A$. Indeed, our argument will show that any
bipartite Borel graph satisfying $\Hall_{\epsilon,1}$ has a Borel
perfect matching provided the vertex set of the graph can be written as a
union of Borel sets $A_n$ of pairwise distance greater than $f(n)$, for $f$
sufficiently large as above. 

We construct by induction an increasing sequence of Borel partial matchings
$M_n$ for $n \in \N$ such that the set of vertices incident to edges in
$M_n$ includes $\bigcup_{m \leq n} A_m$. We will also ensure that $G - M_n$
satisfies $\Hall_{\epsilon_n,f(n)}$, where $\epsilon_n = \epsilon - \sum_{i
\leq n} 8/f(i)$.

For ease of notation we let $M_{-1} = \emptyset$ and $\epsilon_{-1} =
\epsilon$. Note that the hypotheses of the theorem gives us
$\Hall_{\epsilon_{-1},1}$.  
Suppose that we have constructed $M_{n-1}$ as above. We
will now construct $M_{n}$. Let $X_{n-1} \subset X$ be the vertex set of $G
- M_{n-1}$. For each $x \in A_n \inters X_{n-1}$, we can find an edge $e$
incident to $x$ such that $(G - M_{n-1}) - \{e\}$ satisfies Hall's condition.
This is because $G - M_{n-1}$ satisfies Hall's condition, and any edge $e$
from a perfect matching of $G - M_{n-1}$ has this property.
Let $M_{n}'$ be a Borel set containing one such edge $e$ for each
$x \in A_n \inters X_{n-1}$ and let $M_{n} = M_{n-1} \cup M_{n}'$. For
instance, let $\langle T_i \mid i \in \N \rangle$ be a Borel set of
automorphisms generating $G$, and for each $x \in A_n \inters X_{n-1}$ put the edge
$\{x, T_i(x)\} \in M_{n}'$ if $i$ is least such that $T_i(x) \neq x$ and $(G - M_{n-1}) -
\{\{x,T_i(x)\}\}$ satisfies Hall's condition.

To finish the proof we will show $G - M_{n}$ satisfies
$\Hall_{\epsilon_{n},f(n)}$ assuming $G-M_{n-1}$
satisfies $\Hall_{\epsilon_{n-1},f(n-1)}$.
Let $F$ be a finite $(G-M_{n})^{2}$-connected
subset of $B_0 \inters X_{n}$ or $B_1 \inters X_{n}$, where $X_{n} \subset
X$ is the vertex set of $G - M_{n}$. Our rough idea is that if $F$ is large, since $A_n$ is sparse, we lose only a tiny fraction of the neighbors of $F$ in
passing from $G-M_{n-1}$ to
$G-M_n$. If $F$ is small, we use the fact that the
removal of any single edge in $M_n$ from $G-M_{n-1}$ preserves 
Hall's condition. 

Let $D = \Nbhd_{G-M_{n-1}}(F) - \Nbhd_{G-M_{n}}(F)$. Suppose first that
$\card{D} \geq 2$. Each $x \in D$ must be a neighbor of some element $y_x$
of $F$. Each $x \in D$ must also be an element of $A_n$ or a neighbor of an
element of $A_n$ (which it is matched to in $M_n$). Fix distinct $x, x' \in
D$. Then $x$ and $x'$ are associated to distinct elements of $A_n$ since $M_n$ is a
matching and $F$ is a subset of one half of the bipartition of $G$.
Further, since elements of $A_n$
have pairwise distance greater than $f(n)$, we have $d_G(y_{x},y_{x'}) > f(n) - 4$. Since $F$ is connected in $(G -
M_{n})^{2}$, there must be at least $\lfloor (f(n) - 4)/4 \rfloor \geq f(n)
/ 8$ elements of $F$ of distance at most $(f(n) - 4)/2$ from each $y_x$.
Since the closed $(f(n) - 4)/2$-balls around each $y_x$ are disjoint, it
follows that $\card{D} \leq (8/f(n))|F|$ and so
\begin{equation*}
\begin{split}
\card{\Nbhd_{G-M_{n}}(F)} & = \card{\Nbhd_{G-M_{n-1}}(F)} - \card{D}\\
& \geq (1 + \epsilon_{n-1})\card{F} - (8/f(n)) \card{F}\\
& \geq (1 + \epsilon_{n})\card{F}.
\end{split}
\end{equation*}

Suppose now that $\Nbhd_{G-M_{n-1}}(F)$ contains at most one vertex that is
not an element of $\Nbhd_{G-M_{n}}(F)$. First we establish
$\card{\Nbhd_{G-M_{n}}(F)} \geq \card{F}$. Observe that either for some $e
\in M_n'$, we have $\Nbhd_{G-M_{n}}(F) = \Nbhd_{G-M_{n-1}-\{e\}}(F)$ or we
have $\Nbhd_{G-M_{n}}(F)=\Nbhd_{G-M_{n-1}}(F)$. In the first case we have
Hall's condition from the choice of $M_{n}'$ and in the second we have it
by the induction hypothesis. Now if additionally $\card{F} \geq f(n)$, then
since
$\card{\Nbhd_{G-M_{n-1}}(F)} - \card{\Nbhd_{G-M_{n}}(F)} \leq 1$, we have
$\card{\Nbhd_{G-M_{n}}(F)} \geq \card{\Nbhd_{G-M_{n-1}}(F)} -
(1/f(n))\card{F}$,
and we can conclude $\card{\Nbhd_{G-M_{n}}(F)} \geq (1 +
\epsilon_{n})\card{F}$ as above.
\end{proof}

We remark here that an analogue of Theorem~\ref{bairely_hall} for one-sided
matchings is also true. Suppose $G$ is a Borel graph with Borel bipartition
$\{B_0, B_1\}$ satisfying a one-sided version of $\Hall_{\epsilon,1}$. Then
one can use the same type of argument as Theorem~\ref{bairely_hall},
inductively constructing one-sided matchings $M_n$ and verifying that $G -
M_n$ satisfies the one-sided $\Hall_{\epsilon_{n},f(n)}$ condition. 

\section{Paradoxical decompositions}
\label{sec:paradox}

We now exploit a well-known (see for instance~\cite{MR1721355}) connection between matchings and paradoxical
decompositions to prove Theorem~\ref{baire_paradox}. This connection has
been previously used in the setting of descriptive graph combinatorics
in~\cite{GMP} and~\cite{KM}. 

\begin{defn}
Suppose $a$ is an action of a group $\Gamma$ on a space $X$ and $S \subset
\Gamma$ is a finite symmetric set. Let $\Gpdx(a,S)$ be the bipartite
graph with vertex set $\{0,1,2\} \times X$ where there is an edge from
$(i,x)$ to $(j,y)$ if exactly one of $i$ and $j$ is equal to $0$, and there
is a $\gamma \in S$ such that $\gamma \cdot x = y$. 
\end{defn}

Now $\Gpdx(a,S)$
has a perfect matching $M$ if and only if $a$ has a paradoxical
decomposition using group elements from $S$. To see this, suppose $S =
\{\gamma_1, \ldots, \gamma_n\}$ and $M$ is a perfect matching of
$\Gpdx(a,S)$. Now put $x \in A_i$ if $(0,x)$
is matched to $(1,\gamma_i \cdot x)$ and $\gamma_j \cdot x \neq \gamma_i
\cdot x$ for all $j < i$. Similarly, put 
$x \in B_i$ if $(0,x)$ is matched to
$(2,\gamma_i \cdot x)$ and $\gamma_j \cdot x \neq \gamma_i \cdot x$ for all
$j < i$. Then we see $\{A_1, \ldots, A_n,
B_1, \ldots, B_n\}$ partitions the space, as does $\gamma_i A_i$ and also
$\gamma_i B_i$. The converse is proved via the same identification.

\begin{proof}[Proof of Theorem~\ref{baire_paradox}.]

Assume $\Gamma$ acts by Borel automorphisms on a Polish space
$X$ and the action has some paradoxical decomposition using group elements from a finite
symmetric set $S$, so $\Gpdx(a,S)$ satisfies Hall's condition. Now by
enlarging $S$ to $S^2 = \{\gamma \delta : \gamma, \delta \in S\}$, we claim
that $\Gpdx(a,S^2)$ satisfies $\Hall_{1,1}$. This is trivial for
finite sets of the form $F = \{0\} \times F'$ since $\Nbhd_{\Gpdx(a,S)}(\{0\}
\times F') = \{1,2\}
\times F''$ for some $F''$ where $\card{F''} \geq \card{F'}$.

To finish, it suffices to check sets of the form $F = \{1,2\} \times F'$,
since 
\[\Nbhd_{\Gpdx(a,S)}((\{1\}\times F_1) \union (\{2\} \times F_2)) =
\Nbhd_{\Gpdx(a,s)}(\{1,2\} \times (F_1 \union F_2))\] 
and $\card{(\{1\}\times
F_1) \union (\{2\} \times F_2)} \leq \card{\{0,1\} \times (F_1 \union
F_2)}$.
So consider $F = \{1,2\} \times F'$. If we let $\{0\} \times F''
 = \Nbhd_{\Gpdx(a,S)}(\{1,2\} \times F')$, then $\card{F''} \geq
\card{F}$, so
\begin{equation*}
\begin{split}
\card{\Nbhd_{\Gpdx(a,S^2)}(F)} & = \card{\Nbhd_{\Gpdx(a,S^2)}(\{1,2\}
\times F')} \\
& \geq \card{\Nbhd_{\Gpdx(a,S)}(\{1,2\} \times F'')} \\
& \geq \card{\{1,2\} \times F''} \\
& \geq 2 \card{F}.
\end{split}
\end{equation*}

Thus, by Theorem~\ref{bairely_hall}, since $\Gpdx(a,S^2)$ is a bipartite
Borel graph satisfying $\Hall_{1,1}$, it has a Borel perfect
matching on a $\Gpdx(a,S^2)$-invariant comeager Borel set $A$.  
We can use the axiom of choice to
extend to a perfect matching $M$ defined on the whole graph $\Gpdx(a,S^2)$.
Now the paradoxical decomposition we defined above associated to this
matching will have pieces with the Baire property. This is because the set
of $x \in X$ such that $(0,x) \in A$ is Borel and comeager, and
hence the pieces of the paradoxical decomposition are relatively Borel in a
comeager Borel set.
\end{proof}

Note that we needed to increase the number of pieces in our paradoxical
decomposition in the proof above to obtain a paradoxical decomposition
using pieces with the Baire property; this occurs when we enlarge the set
$S$ of group elements to $S^2$. This is known to be necessary in general 
by a result of Wehrung~\cite{MR1209101}. In particular,
Wehrung has shown that a Baire measurable paradoxical decomposition of the $2$-sphere
by isometries must use at least six pieces, while there are
paradoxical decompositions using pieces without the Baire property with
only four pieces.

\section{Equidecomposability} \label{equidec}

Suppose $a$ is an action of a group $\Gamma$ on a space $X$ and $A$ and $B$
are subsets of $X$. Then for any finite symmetric $S \subset \Gamma$ we may
form a bipartite graph $G(a,S,A,B)$ on the space $\{0\} \times A \union
\{1\} \times B$ where $(i,x) \mathrel{G(a,S,A,B)} (j,y)$ if $i \neq j$ and there is
an $\gamma \in S$ such that $\gamma \cdot x = y$. It is obvious that $A$
and $B$ are $a$-equidecomposable if and only if there is some finite
symmetric $S \subset \Gamma$ such that $G(a,S,A,B)$ has a perfect matching.
From this perspective, the graph $\Gpdx(a,S)$ used in Section~\ref{sec:paradox} is really the
graph $G(b,S,\{0\} \times X, \{1,2\} \times X)$ where $b$ is the action of 
$\Z/3\Z \times \Gamma$ on $\{0,1,2\} \times
X$ via $(n,\gamma) \cdot_b (i,x) = (n+i, \gamma \cdot_a x)$. Of
course,
$a$ has a paradoxical decomposition if and only if $\{0\} \times X$ is
$b$-equidecomposable with $\{1,2\} \times X$. 

We begin by proving a lemma relating matchings in $2$-regular acyclic Borel
bipartite graphs with certain graphs containing them. If $G$ is a
$2$-regular acyclic graph on $X$ and $n \geq 1$, let $G_n$ be the
Borel graph on $X$ where $x \mathrel{G_n} y$ if there is a $G$-path of odd
length at most $2n-1$ from $x$ to $y$. Note that $G_1 = G$, and $G_{n}$ is
a $2n$-regular graph. 

\begin{lemma}\label{2-regular_matching_lemma}
  Suppose $G$ is a $2$-regular acyclic Borel bipartite graph with Borel
  bipartition $\{B_0, B_1\}$ and $n \geq 1$. Then if $G_n$ has a Borel perfect matching on a $G_n$-invariant Borel set $Y$, then $G$ has a
Borel perfect matching on $Y$. 
\end{lemma}
\begin{proof}
Since
$G$ is $2$-regular and acyclic, using the axiom of choice, we may linearly
order each connected component of $G$ so that each vertex $x$ has
one neighbor greater than $x$ and one neighbor less than $x$. Note that a
perfect matching of $G$ is characterized by the property that for each
connected component $C$ of $G$, either every $x \in C \inters B_0 $ is
matched to its unique neighbor less than $x$, or every $x \in C \inters
B_0$ is matched to its unique neighbor greater than $x$. The idea of the
proof is that from a matching in $G_n$ we can select this unique direction
in $G$ by ``averaging'' the direction of matched edges in $G_n$ in a ball
around each point. 

Suppose that $G_n \restriction Y$ has a Borel perfect matching given by the Borel
function $M \from B_0 \inters Y \to B_1 \inters Y$. We will use $M$ to
define a Borel perfect
matching $M'$ of $G \restriction Y$. If $x \in B_0$, let $D_{n}(x) = \{y \in B_0 :
d_G(x,y) \leq 2n - 2\}$. 
Note that $\card{D_{n}(x)} = 2n-1$. Define $M' \from B_0 \inters Y \to B_1
\inters Y$ by
setting $M'(x)$ to be the unique $G$-neighbor $y$ of 
$x$ such that the connected component of $y$ in $G \restriction X \setminus
\{x\}$ contains at least $n$ elements of $M(D_n(x))$. In terms of our order,
$M'(x)$ will be the neighbor of $x$ that is less than $x$ 
if and only if $M(D_n(x))$ contains at least $n$ elements less than $x$. 
We now claim that for every
$x, y \in B_0$ in the same connected component of $G$, $M'(x) < x$ if and
only if $M'(y) < y$. Hence, $M'$ is a   
perfect matching of $G$.
Our claim is easy to see when $D_{n}(x)$ and $D_{n}(y)$ are disjoint since
if $z \in B_0$ is in the same connected component as $x$ and $y$ but $z \notin
D_n(x) \union D_n(y)$, then $x < M(z) < y$ if and only if $x < z < y$. Our
claim then follows from the fact that $M$ is a bijection. 
To finish, note that when $D_n(x)$ and $D_n(y)$ are not disjoint, then we
can find some $y'$ where $D_n(y')$ is disjoint from both $D_n(x)$ and
$D_n(y)$. 
\end{proof}

We now prove a result which shows we cannot generalize
Theorem~\ref{baire_paradox} to say that equidecomposable Borel sets have
Baire measurable equidecompositions. 

\begin{thm}
  There is an action of a group $\Gamma$ on a Polish space $X$ by Borel
  automorphisms and
  Borel sets $B_0, B_1 \subset X$ that are equidecomposable, but not
  equidecomposable using pieces with the Baire property.
\end{thm}
\begin{proof}
Let $G$ be the graph of Laczkovich~\cite{MR947676}, which is a $2$-regular
acyclic bipartite Borel graph on a Polish space $X$ with a bipartition into
Borel sets $\{B_0, B_1\}$ so that $G$ has no Borel perfect matching
restricted to any comeager Borel set. For this graph, we may find finitely
many Borel involutions $\langle T_i \mid i < n \rangle$ generating the
edges of $G$. Such $T_i$ are easy to define explicitly for Laczkovich's
graph, and more generally, such involutions always exist for any bounded
degree Borel graph (see the remarks after \cite[Proposition
4.10]{MR1667145}). These $T_i$ define an action $a$ of the group $\F_n$ on
$X$ by Borel automorphisms. $G$ has a perfect matching since it is
$2$-regular and acyclic. Hence, $B_0$ and $B_1$ are $a$-equidecomposable.

Suppose $S$ is a finite set of
group elements in $\F_n$ of word length at most $2m-1$, and $B_0$ and $B_1$
are $a$-equidecomposable using pieces with the Baire property and group
elements from $S$. Then the graph $G(a,S,B_0,B_1)$ defined above must have a perfect
matching $M$ which is Borel on a $G$-invariant comeager Borel set. 
This is because every set with the Baire property differs from an open
set by a meager set, and every comeager set contains a Borel comeager (indeed,
$G_\delta$) set. Finally, every comeager Borel subset of $X$ contains a comeager
$G$-invariant Borel set since $G$ is generated by homeomorphisms and the
image of a meager set under a homeomorphism is meager. Since 
$G(a,S,B_0,B_1)$ is a subset of $G_m$, this would mean $G_m$ has a Borel
perfect matching on an comeager invariant Borel set. But by 
Lemma~\ref{2-regular_matching_lemma}, this would mean $G$ has a Borel
perfect matching on a comeager invariant Borel set which is a contradiction.
\end{proof}

Laczkovich has solved Tarski's circle squaring problem, showing that a
circle and square of the same area in the plane are equidecomposable by
rigid motions~\cite{MR1037431}. It is an open problem whether the pieces
used in such a decomposition may have the Baire property.\footnote{This
question has now been resolved by in the affirmative by Grabowski,
M\'ath\'e, and Pikhurko~\cite{GMP2}.}

\section{A Baire category solution to the dynamical von Neumann-Day problem}

Suppose that $a$ is an action of a group
$\Gamma$ on a space $X$. Say that a function $f \from X \to X$ is
\define{$a$-Lipschitz} if there is a finite set of group elements $S
\subset \Gamma$ such that for every $x \in X$ there is a $\gamma \in S$
such that $f(x) = \gamma \cdot x$. Say a graph $G$ on $X$ is
\define{$a$-Lipschitz} if there is a finite set of group elements $S
\subset \Gamma$ such that for every edge $\{x,y\} \in G$ there is a
$\gamma \in S$ such that $y = \gamma \cdot x$.

We now prove a lemma which closely parallels Whyte's geometric solution to
the von Neumann-Day problem.

\begin{lemma}\label{4-regular tree}
  Suppose that $a$ is a nonamenable action of a group $\Gamma$ on a Polish
  space $X$ by Borel
  automorphisms. Then there exists a $4$-regular acyclic
  $a$-Lipschitz graph $G$ on $X$ 
  and a comeager $G$-invariant Borel set $A \subset X$ so that $G \restriction A$ is Borel. 
\end{lemma}
\begin{proof}
  Since the action of $a$ is nonamenable and hence paradoxical, we may find
  a finite symmetric set $S \subset \Gamma$ so
  that $\Gpdx(a,S)$ satisfies Hall's condition and $\Gpdx(a,S^2)$ satisfies
  $\Hall_{1,1}$ as in the proof of Theorem~\ref{baire_paradox}. Let $\Gpdx'(a,S^2)$ be the graph on
  $\{0,1,2,3\} \times X$ where there is an edge between $(i,x)$ and $(j,y)$
  if exactly one of $i$ and $j$ is equal to $0$ and there is a $\gamma \in
  S^2$ such that $\gamma \cdot x = y$. Since $\Gpdx(a,S^2)$ satisfies
  $\Hall_{1,1}$, the new graph $\Gpdx'(a,S^2)$ satisfies $\Hall_{1/3,1}$.
  Thus, by Theorem~\ref{bairely_hall}, we may find
  a perfect matching of $\Gpdx'(a,S^2)$ which is Borel on a comeager
  invariant Borel set $A'$. Now define functions $f_0, f_1, f_2 \from X \to
  X$ by setting $f_i(x)$ to be the unique $y$ such that $(i+1,x)$ is
  matched to $(0,y)$. These functions are $a$-Lipschitz, injective, and
  have disjoint ranges which partition $X$. Let $A = \{x :
  (0,x) \in A'\}$ which is comeager and Borel. It is invariant under $f_0$,
  $f_1$, and $f_2$ since $A'$ is $\Gpdx'(a,S^2)$-invariant and $S^2$
  contains the identity and so $(0,x) \in A'$ if and only if $(n,x) \in A'$
  for every $n$. Note that the functions $f_i$ are Borel when restricted to
  $A$. We will define a $4$-regular tree in terms of these functions.

  Consider the graph $H$ generated by the three functions $f_0$, $f_1$ and
  $f_2$, so $x \mathrel{H} y$ if there is an $i$ such that $f_i(x) = y$ or
  $f_i(y) = x$. Each connected component of $H$ contains at most one cycle.
  To see this, observe that since the $f_i$ have disjoint ranges and
  generate $H$, any cycle must arise from some $x$ and $i_0, \ldots, i_n$
  where $f_{i_n} \circ \ldots \circ f_{i_0}(x) = x$. Further, for such an
  $x$, every other element in the same connected component is in the
  forward orbit of $x$ under the $f_i$. Hence, another cycle of this form
  would contradict the fact that the $f_i$ are injective and have disjoint
  ranges.

  We may assume that any cycle in $H$ is of the form $f_0 \circ \ldots
  \circ f_0(x) = x$, so all the edges in the cycle arise from the function
  $f_0$. This is because given any $f_i$ as above, we can replace them with
  the functions $g_i$, for $i \in \{0,1,2\}$ which are defined as follows. If
  $\{x, f_j(x)\}$ is an edge in a cycle in $H$, then define $g_0(x) =
  f_j(x)$, $g_1(x) = f_{j+1 \bmod{3}}(x)$ and $g_2(x) = f_{j + 2
  \bmod{3}}(x)$. If $x$ is not contained in any cycle, then set $g_i(x) =
  f_i(x)$. Then $g_0$, $g_1$, and $g_2$ will be injective, have disjoint
  ranges that partition $X$, and any cycle in the graph generated by the
  $g_i$ will contain edges generated only by the function $g_0$. 
  
  Now it is easy to define a $4$-regular acyclic $a$-Lipschitz graph $G$ on
  $X$ with the same connected components as $H$ and so that if $x
  \mathrel{T} y$, then $x$ and $y$ have distance at most $2$ in $H$. We give
  one such construction. 

  Define functions $f_0'$, $f_1'$ and $f_2'$ on $X$ as follows. Fix a Borel
  linear ordering of $X$. Suppose $x$ is
  in a connected component containing a unique cycle $x_0, x_1, \ldots,
  x_n, x_0$, with $f_0(x_i) = x_{i+1}$ and $f_0(x_n) = x_0$, and where $x_0$ is
  the least element of $x_0, \ldots, x_n$ under the Borel linear ordering. Now 
  if $x$ is of the form
  $f^k_1(x_m)$ for some $m \leq n$ and $k \geq 0$, define $f_0'(x) =
  f_{0} \circ f_{1}(x)$, $f_1'(x) = f_{1}(x)$, and $f_2'(x) = f_{2}(x)$. For $x$ not of this form, or in connected
  components not containing cycles, define $f_0'(x) = f_0(x)$, $f_1'(x) =
  f_1(x)$, and $f_2'(x) = f_2(x)$. Note that the $f_i'$ are injective, have
  disjoint ranges, and $\bigunion_i \ran(f_i')$ is the set of $x$ not contained
  in cycles of $H$. Let $G$ be the graph where $x \mathrel{G} y$ if there
  is an $i$ such that $f_i'(x) = y$ or $f_i'(y) = x$, or if $x$ and $y$ are
  contained in the unique cycle $x_0, x_1, \ldots, x_n, x_0$ of a connected
  component of $H$, and there is an $i$ such that $\{x,y\} =
  \{x_i,x_{i+1}\}$. 
  
  Finally, our construction of $G$ is Borel on $A$, since the $f_i$ are
  Borel on $A$. 
\end{proof}

To finish the proof of Theorem~\ref{dynamical_vnd} it suffices to prove the
following lemma. 

\begin{lemma}\label{F_2 action}
  Suppose that $G$ is an acyclic $4$-regular Borel graph on a comeager
  Borel subset $C$ of a Polish space $X$. 
  Then there is a free Borel action of $\F_2 =
  \langle a, b \rangle$ on a $G$-invariant comeager Borel set $A \subset C$
  generating $G \restriction A$. 
\end{lemma}
\begin{proof}
  Apply Lemma \ref{comeager} to obtain $A = \bigcup_{n \in \N} A_n \subset
  C$, which
  is invariant and comeager and where points in $A_n$ are pairwise of distance
  greater than $16 \cdot 4^{n}$. It suffices to construct two Borel automorphisms
  $f_1, f_2$ of $A$ so that for every $x \in A$, $f_1(x) \neq f_2(x)$
  and $\{x,f_1(x)\}, \{x,f_2(x)\} \in G$. We will construct $f_1$ and $f_2$
  and their inverse functions $f_{-1} = f_1^{-1}$ and $f_{-2} = f_2^{-1}$
  as increasing unions of injective Borel partial functions $\langle
  f_{i,n} \mid n \in \N \rangle$, so $f_i = \bigunion_{n\in \N} f_{i,n}$ for
  each $i \in \{-2,-1,1,2\}$. For each $n$, 
  these partial functions $f_{i,n}$ will all have the same domain $D_n$,
  where 
  \begin{enumerate}
  \item $A_n \subset D_n$.
  \item If $x,y \in D_n$ and $d(x,y) \leq 4$, then $x$ and $y$ are
  connected in $G \restriction D_n$.
  \item $G^{\leq 8} \restriction D_n$ has finite connected
  components of diameter at most $4^{n}$, where $x \mathrel{G^{\leq 8}} y$
  if and only if $1 \leq d_G(x,y) \leq 8$.
  \end{enumerate}
  For ease of notation, let $f_{i,-1} = \emptyset$ for all $i \in \{-2,-1,1,2\}$.

  Given Borel partial functions $f_{i,n}$ with domain $D_n$, we construct
  the functions $f_{i,n+1}$ as follows. For each $x \in A_{n+1} \setminus
  D_n$, we define the set $D_{n+1,0}(x) = \{x\}$, and recursively let
  $D_{n+1,k+1}(x)$ be the set of points $y$ that are adjacent to $D_{n+1,k}(x)$ and
  not contained in $D_n$ or $\bigunion_{j \leq k} D_{n+1,j}(x)$ such that
  $y$ is distance at most $3$ from $D_n$. Let $D_{n+1}(x) = \bigunion_j
  D_{n+1,j}(x)$. We can think of $D_{n+1}(x)$ as being comprised of
  (overlapping) paths
  which start at $x$ and end at elements of $D_n$. We will define $D_{n+1}$
  to be the union of $D_n$ and $\bigunion_{x \in A_{n+1} \setminus D_n}
  D_{n+1}(x)$. It is clear that $D_{n+1}$ is Borel and satisfies (1). 
  
  We first check that $D_{n+1}$ satisfies properties (2) and (3). Suppose
  $x \in A_{n+1} \setminus D_n$. Since $x$ cannot be of $G$-distance less
  than or equal to $4$ 
  from two different connected components of $G^{\leq 8}$, there can
  be most one connected component $C$ of $G^{\leq 8} \restriction D_n$ with
  $d_G(x,C) \leq 4$. 
  If there is such a $C$, then inductively,
  all elements of $D_{n+1}(x)$ must be $G$-distance at most $4$ from this $C$.
  From this, we can see that the connected component $C'$ of $x$ in the graph $G^{\leq 8}
  \restriction D_n \union D_{n+1}(x)$ is a subset of the union of 
  $\Nbhd_{G^{\leq 8}}(C)$ with all connected components of $G^{\leq
  8}$ that are $G^{\leq 8}$-adjacent to $\Nbhd_{G^{\leq 8}}(C)$.
  Thus, the $G^{\leq
  8}$-diameter of $C'$ is at most 
  $3 \cdot 4^n + 4 \leq 4^{n+1}$, since connected components of
  $G^{\leq 8} \restriction D_n$ have diameter at most $4^n$.
  If there is no such $C$, then the connected component of $x$ in the graph
  $G^{\leq 8} \restriction D_n \union D_{n+1}(x)$ consists of $x$ together
  with all connected components of $G^{\leq 8}$ that are $G^{\leq
  8}$-adjacent to $x$. The $G^{\leq 8}$-diameter of this connected
  component is therefore at most $2 \cdot 4^n + 2 \leq 4^{n+1}$. 
  Since distinct elements of $A_{n+1}$ are of $G^{\leq 8}$-distance greater
  than $2 \cdot 4^{n+1}$ we see that each $x \in A_{n+1} \setminus D_n$ is an element of
  a distinct connected component of $G^{\leq 8} \restriction D_{n+1}$.
  Thus, the set $D_{n+1}$ satisfies (2) and (3). 
  
  We will extend $f_{i,n}$ to $f_{i,n+1}$ in finitely many steps by letting
  $f_{i,n+1,0} = f_{i,n}$ and iteratively extending $f_{i,n+1,k}$ to
  $f_{i,n+1,k+1}$ so that its domain includes $D_{n+1,k} = \bigunion_{x \in
  A_{n+1}} D_{n+1,k}(x)$. In particular,
  we will let $f_{i,n,k+1}$ be a partial Borel function extending
  $f_{i,n,k}$ such that for each $x \in D_{n+1,k}$:
  \begin{enumerate}
  \item If there is a neighbor
  $y$ of $x$ such that $f_{-i,n,k}(y) = x$, then $f_{i,n,k+1}(x) = y$.
  \item The values of $f_{i,n,k+1}(x)$ over $i \in
  \{-2,-1,1,2\}$ are distinct neighbors of $x$ such that $f_{i,n,k+1}(x)
  \notin \ran(f_{i,n,k})$. 
  \end{enumerate}
  There
  will be at most one neighbor $y$ of $x$ from $D_{n+1,k}$ for which
  $f_{i,n,k}(y) = x$ for some $i$. Further, there will at most one point in $D_n$ of
  distance $1$ from $x$ by condition (2) on $D_{n}$. If there is no point
  of $D_n$ of distance $1$ from $x$, then likewise there is at most one
  point in $D_n$ of distance $2$ from $x$ by condition (2) on $D_{n}$. 
  Thus, it is possible to define 
  the $f_{i,n,k+1}(x)$ and still satisfy our requirements (1) and (2) for
  the $f_{i,n,k+1}(x)$. 
  This is because there are at most two values of $i \in \{-2,-1,1,2\}$ for which
  $f_{i,n,k}(x)$ is determined by requirement (1). And if there 
  is a
  $z \in D_n$ of distance $2$ from $x$, there is only one $j$ with $f_{j,n,k}(z)$ of distance $1$
  from $x$, so at least one remaining $i' \in \{-2,-1,1,2\}$ with $i' \neq
  j$ for which we
  can set $f_{i',n,k+1}(x) = f_{j,n,k}(z)$ and so satisfy requirement (2).
\end{proof}

It is not the case that every acyclic Borel graph of degree $4$ 
is generated by a free Borel action of $\F_2$. This follows from results 
in \cite{Marks}. Hence, discarding a meager set is necessary.

We may conclude Theorem~\ref{dynamical_vnd} by combining
Lemma~\ref{4-regular tree} and Lemma~\ref{F_2 action}. Indeed, we obtain
the stronger result that there is a Borel comeager set $B \subset X$ such
that the action of $\F_2$ is invariant and Borel on $B$. 

\bibliography{references}

\end{document}